\documentclass[11pt]{amsart}

\usepackage{amssymb,hyperref}
\usepackage[all]{xy}

\newtheorem{theorem}{Theorem}

\newtheorem{conjecture}[theorem]{Conjecture}
\newtheorem{corollary}[theorem]{Corollary}
\newtheorem{definition}[theorem]{Definition}
\newtheorem{example}[theorem]{Example}

\newtheorem{fact}[theorem]{Fact}
\newtheorem{lemma}[theorem]{Lemma}
\newtheorem{problem}[theorem]{Problem}
\newtheorem{proposition}[theorem]{Proposition}
\newtheorem{question}[theorem]{Question}
\newtheorem{remark}[theorem]{Remark}

\newcommand{\bcon}{\begin{conjecture}}
\newcommand{\econ}{\end{conjecture}}
\newcommand{\bcor}{\begin{corollary}}
\newcommand{\ecor}{\end{corollary}}
\newcommand{\bdf}{\begin{definition}}
\newcommand{\edf}{\end{definition}}
\newcommand{\beq}{\begin{equation}}
\newcommand{\eeq}{\end{equation}}
\newcommand{\bexa}{\begin{example}}
\newcommand{\eexa}{\end{example}}
\newcommand{\bfac}{\begin{fact}}
\newcommand{\efac}{\end{fact}}
\newcommand{\blem}{\begin{lemma}}
\newcommand{\elem}{\end{lemma}}
\newcommand{\bprb}{\begin{problem}}
\newcommand{\eprb}{\end{problem}}
\newcommand{\bpro}{\begin{proposition}}
\newcommand{\epro}{\end{proposition}}

\newcommand{\bque}{\begin{question}}
\newcommand{\eque}{\end{question}}
\newcommand{\brem}{\begin{remark}}
\newcommand{\erem}{\end{remark}}
\newcommand{\bthm}{\begin{theorem}}
\newcommand{\ethm}{\end{theorem}}
\newcommand{\bmat}{\begin{matrix}}
\newcommand{\emat}{\end{matrix}}

\newcommand{\bpr}{\begin{proof}}
\newcommand{\epr}{\end{proof}}

\newcommand{\SL}[1]{\operatorname{SL(#1,\C)}}

\newcommand{\GL}[1]{\operatorname{GL(#1,\C)}}
\newcommand{\SO}[1]{\operatorname{SO(#1,\C)}}
\newcommand{\Sp}[1]{\operatorname{Sp(#1,\C)}}

\newcommand{\lb}{\label}
\newcommand{\la}{\langle}
\newcommand{\ra}{\rangle}
\newcommand{\comment}[1]{\,}

\newcommand{\cal}{\mathcal}

\newcommand{\Z}{\mathbb Z}

\newcommand{\C}{\mathbb C}
\newcommand{\tr}{\mathrm{tr}}
\newcommand{\Ch}{\mathcal{C}h}

\newcommand{\X}{\mathcal X}

\newcommand{\hm}{\mathrm{Hom}}
\newcommand{\aq}{/\!\!/}

\title{Varieties of Characters}

\begin{document}

\thispagestyle{empty}

\begin{abstract}
Let $G$ be a connected reductive affine algebraic group. In this short note we define the {\it variety of $G$-characters} of a finitely generated group $\Gamma$ and show that the quotient of the $G$-character variety of $\Gamma$ by the action of the trace preserving outer automorphisms of $G$ normalizes the variety of $G$-characters when $\Gamma$ is a free group, free abelian group, or a surface group.
\end{abstract}

\author[S. Lawton]{Sean Lawton}

\address{Department of Mathematical Sciences, George Mason University,
4400 University Drive,
Fairfax, Virginia  22030, USA}

\email{slawton3@gmu.edu}

\author{Adam S. Sikora}

\address{244 Math Bldg, University at Buffalo, SUNY, Buffalo, NY 14260}
\email{asikora@buffalo.edu}

\subjclass[2010]{14D20; 14L30; 20C15.}

\keywords{Character variety, trace variety, variety of characters.}

\pagestyle{myheadings}

\maketitle

\section{Introduction}
For a finitely generated group $\Gamma$ and a connected reductive affine algebraic group $G$ over the complex numbers, the {\it $G$-character variety of $\Gamma$} is the Geometric Invariant Theory (GIT) quotient $\X(\Gamma, G):=\hm(\Gamma, G)\aq G$ of the $G$-representation space of $\Gamma$, $\hm(\Gamma, G)$, (considered as an algebraic set) by the adjoint action of $G$. 

This algebraic set is known for its relationship to moduli spaces of geometric structures (\cite{CG2, G6,  GM2,JM, KM,Thurston}), 3 and 4 dimensional topology (\cite{CCGLS, CS, BZ, PrSi, Bu, Cur}),  mathematical physics (\cite{At, AtBo, Borel-Friedman-Morgan, Hi, Kac-Smilga, Wi, JW, KW}) and invariant theory (\cite{BH, La3, Si2}).

The term ``$G$-character variety'' is motivated by the fact that its closed points for $G=\SL{n}, \Sp{n},\SO{2n+1}$ are classified by the $G$-characters of $\Gamma$. (Perhaps the first proof of this was given by Culler-Shalen in \cite{CS} for the case $\SL{2}$.) However, it is not the case for $\SO{2n}$, see \cite[Appendix A]{FlLa2b}, \cite{Si7, Si8}.
For that reason, it is useful to consider the space of characters of $\Gamma$, which we show to have a natural structure of an algebraic set, leading to the notion of the {\it variety of characters}.

In this short note we first make clear the relationship between the variety of characters of $\Gamma$ and the character variety of $\Gamma$ by showing that although the variety of characters is not generally isomorphic to the character variety, for many groups (including free groups, free abelian groups, and surface groups), the quotient of the $G$-character variety by a finite group (the trace-preserving outer automorphisms of $G$) normalizes the variety of characters (see Corollary \ref{cor-normal}). As a consequence, we conclude that regular functions on the character varieties as above are rational functions in characters. 

In the last section we compute the the trace-preserving outer automorphisms of $G$ for various simple complex Lie groups.

\subsection*{Acknowledgements}
We acknowledge support from U.S. National Science Foundation grants: DMS 1107452, 1107263, 1107367 ``RNMS: GEometric structures And Representation varieties" (the GEAR Network) and Grant No. 0932078000 while we were in residence at the Mathematical Sciences Research Institute in Berkeley, California, during the Spring 2015 semester. Lawton was also partially supported by the grant from the Simons Foundation (Collaboration \#245642) and the U.S. National Science Foundation (DMS \#1309376). Lastly, we thank the referee for helping improve the paper.

\section{Variety of Characters}\lb{varchar}

Let $\Gamma$ be a finitely generated group. Fix a connected reductive subgroup $G$ of $\SL{n}$.  A {\it $G$-character of $\Gamma$}, denoted $\chi_\rho$, is the trace of a $G$-representation of $\Gamma$: $$\xymatrix{\Gamma\ar[r]^{\rho} & G\ \ar@{^{(}->}[r]&\SL{n}\ar[r]^{\ \ \ \tr}& \C\\}.$$ Denote the set of $G$-characters of $\Gamma$ by ${\mathcal Ch}(\Gamma,G).$

Let $\hm(\Gamma,G)$ and $\mathcal X(\Gamma,G)=\hm(\Gamma,G)\aq G$ be the representation variety and the $G$-character variety, respectively (considered as affine varieties). Let ${\cal T}(\Gamma,G)$ be the trace $G$-algebra of $\Gamma,$ that is, the subalgebra of  $\C[\X(\Gamma,G)]$ generated by trace functions $\tau_\gamma$ for $\gamma\in \Gamma$ defined by $\tau_\gamma(\rho)=\tr(\rho(\gamma))$, see \cite{CS,FlLa2b,Si5} for more details.

It was proven in \cite[Appendix A]{FlLa2b} and in \cite[Thm. 5]{Si5} that ${\cal T}(\Gamma,G)$ coincides with $\C[\X(\Gamma,G)]$ for any $\Gamma$ and $G$ equal to $\SL{m},$ $\Sp{2m},$ or $\SO{2m+1}.$ For the reader's convenience, we include a concise proof of that result here. 

\begin{proposition}\label{xt-prop}
For $G$ equal to $\SL{m}, \Sp{2m},$ or $\SO{2m+1}$, and for all finitely generated $\Gamma$, ${\cal T}(\Gamma,G)=\C[\X(\Gamma,G)].$
\end{proposition}

\begin{proof}
If $\Gamma$ is a group on $r$ generators, then the epimorphism $F_r\to \Gamma$ mapping the free generators of $F_r$ to the generators of $\Gamma$ yields an embedding $\X(\Gamma,G)\to \X(F_r,G)$ and the corresponding dual epimorphism $\C[\X(F_r,G)]\to \C[\X(\Gamma,G)]$. Since that map sends trace functions to trace functions, it is enough to show that $\C[\X(F_r,G)]=\C[G^r]^G$ is generated by trace functions.

By the assumptions of the theorem, $G$ is an algebraic subgroup of $\SL{n}$ for some $n$ and, hence, of the space of the $n\times n$ complex matrices $Mat(n,\C).$ The group $G$ acts on it by conjugation and the embedding $G\subset \SL{n}\subset Mat(n,\C),$ induces a $G$-equivariant epimorphism $\C[Mat(n,\C)^r]\to \C[G^r]$, which restricts to an epimorphism $\C[Mat(n,\C)^r]^G\to \C[G^r]^G$ by the existence of Reynolds operators. Now the statement follows from the results of  \cite{P1} stating that $\C[Mat(n,\C)^r]^G$ is generated by traces of monomials in matrices and their inverses for $G=\SL{m}, \SO{2m+1},$ and $\Sp{2m}$.
\end{proof}

The above statement does not hold for $G=\SO{2m},$ $m\geq 1.$ In that case $\C[\X(\Gamma,G)]$ is a $\cal T(\Gamma,G)$-algebra finitely generated by expressions involving Pfaffians, \cite{ATZ-SO}, \cite[Proposition 16]{Si7}.  

In this paper we intend to elaborate on the relationship between coordinate rings of character varieties and their trace algebras.

Denote by $\varphi$ the natural projection map $$\varphi: \hm(\Gamma,G)\aq{\cal N}(G)\to \mathcal X(\Gamma,\SL{n}),$$
where ${\cal N}(G)$ is the normalizer of $G$ in $\SL{n}.$ Since a normalizer of a connected reductive group is reductive, the above GIT quotient is well defined. By \cite[Corollary 2 to Theorem 1]{Vi}, $\varphi$ is a finite morphism and, hence, its image is closed.

\blem\lb{isolemma}
There is an isomorphism
$$\psi: {\cal T}(\Gamma,G)\to \C[\varphi(\hm(\Gamma,G)\aq\mathcal N(G))]$$
sending regular functions $\tau_\gamma$ to regular functions $\tilde{\tau}_\gamma$ on $\varphi(\hm(\Gamma,G)\aq\mathcal N(G))$ defined by $\tilde{\tau}_\gamma([\rho])= \tr(\rho(\gamma)).$ Since $\tau_\gamma$ generate ${\mathcal T}(\Gamma,G)$, this condition defines $\psi$ uniquely.
\elem

\bpr  Since any polynomial identity in $\tau_\gamma$'s on $\mathcal X(\Gamma,G)$
holds if and only if it holds on $\varphi(\hm(\Gamma,G)\aq{\cal N}(G))$ as well, the above homomorphism is well defined and injective. Since $\C[\mathcal X(\Gamma,\SL{n})]$ is generated by trace functions, $\C[\varphi(\hm(\Gamma,G)\aq \mathcal N(G))]$ is generated by trace functions as well, and hence, $\psi$ is onto.
\epr

We have a map $$\Psi: {\mathcal Ch}(\Gamma,G)\to Spec\, {\cal T}(\Gamma,G)$$ sending $\chi_\rho$ to a algebra homomorphism ${\cal T}(\Gamma,G)\to \C,$ determined by $\tau_\gamma\mapsto \chi_\rho(\gamma).$

It is not hard to see that $\Psi$ is well defined. Furthermore, we have:

\bpro \label{Psi-bij}
$\Psi$ is a bijection.
\epro

\begin{proof}
If $\Psi(\chi_\rho)=\Psi(\chi_{\rho'})$ then $\chi_\rho(\gamma)=\chi_{\rho'}(\gamma)$ for all $\gamma\in \Gamma$, implying $\chi_\rho=\chi_{\rho'}.$  Thus, $\Psi$ is injective.

By Lemma \ref{isolemma}, it remains to prove that
$$\Psi: {\mathcal Ch}(\Gamma,G)\to \varphi(\hm(\Gamma,G)\aq\mathcal N(G))$$
is surjective.  Indeed, for every $\rho\in \hm(\Gamma,G)$, $\Psi(\chi_\rho)=\varphi(\rho)$ and, hence, $\Psi$ is onto.
\end{proof}

The above proposition provides for the structure of an algebraic set on ${\mathcal Ch}(\Gamma,G)$ with its coordinate ring being ${\cal T}(\Gamma,G).$  By analogy with the $G$-character variety of $\Gamma$, we will call ${\mathcal Ch}(\Gamma,G)$ the {\it variety of $G$-characters of $\Gamma$} or simply the {\it variety of characters} when the context is clear.

By Lemma \ref{isolemma}, $\varphi(\hm(\Gamma,G)\aq\mathcal N(G))=\Ch(\Gamma,G)$ and, hence, $\varphi$ can be written as $$\varphi: \hm(\Gamma,G)\aq\mathcal N(G)\to \Ch(\Gamma,G).$$

Following \cite{AB}, denote the set of {\it trace-preserving automorphisms of $G$} by $Aut_T(G)$, that is, the automorphisms $\alpha\in Aut(G)$ such that for all $g\in G$, $\tr(\alpha(g))=\tr(g)$.  Note that $Aut_T(G)$  acts naturally on $\hm(\Gamma,G)$ by $(\alpha,\rho)\mapsto \alpha\circ \rho$ and that this action descends to an action of $Aut_T(G)/Inn(G)$ on $\X(\Gamma, G)$, where $Inn(G)$ is the inner automorphism group of $G$. We have a natural map $$\pi: {\mathcal N(G)}\to Aut_T(G)$$ given by $h\mapsto C_h$ where $C_h(g)=hgh^{-1}$.  Since $Ker(\pi)$ contains the center of $\SL{n}$, $\pi$ is not one-to-one.

Recall that $\rho: \Gamma\to G$ is irreducible if its image is not properly contained in any parabolic subgroup of $G$, see for example \cite{Si4}.

\blem\lb{finite}
$\pi$ is onto.
\elem

\begin{proof}
Let $G_c$ be a maximal compact subgroup of $G$. By \cite[Lemma 1.8]{Gelander}, there exists a representation $\rho: F_2\to G_c$ with a dense image. Then the image of $\rho$ in $G$ is Zariski dense and, in particular, $\rho: F_2\to G_c\hookrightarrow G$ is irreducible. Let $\alpha\in Aut_T(G)$. Since $\C[\X(F_2,\SL{n})]=\cal T(F_2,\SL{n})$ and $\rho$ and $\alpha\rho$ have the same character, they coincide in $\mathcal X(F_2,\SL{n}).$ Since $G$ is a reductive subgroup of $\SL{n}$, $\rho$ and $\alpha\rho$ are completely reducible representations in $\SL{n}$ and, therefore, their $\SL{n}$-conjugation orbits are closed in $\hm(F_2,\SL{n}),$ see  \cite[Theorem 1.27]{L-M}. Since there is a unique closed orbit in each equivalence class in $\mathcal X(F_2,\SL{n})$, see for example \cite[Corollary 6.1]{Do}, $\rho$ and $\alpha\rho$ are conjugate in $\SL{n}$.  Thus, since $\rho$ has Zariski dense image in $G$, this implies that $\alpha: G\to G$ coincides with conjugation of $G$ by some element of $\SL{n}$.
\end{proof} 

We denote $$Out_T(G) := Aut_T(G)/Inn(G),$$ and call this group the {\it trace-preserving} outer automorphisms.  

\blem
$Out_T(G)$ is finite.
\elem

\bpr
$Out_T(G)$ is a subgroup of $Out(G)$ and, consequently, of $Out(\mathfrak g)$ which coincides with the automorphism group of the Dynkin diagram of $\mathfrak g.$ Consequently, $Out_T(G)$ is finite for semisimple $G.$ 

For the general not semisimple case, note that $Out_T(G)$ is the epimorphic image of $\mathcal N(G)/G$ which is finite by \cite[Corollary 3b]{Vi}. 
\epr

Clearly $Out_T(G)$ descends to an action on $\X(\Gamma, G)$ yielding an isomorphism $$\hm(\Gamma,G)\aq\mathcal N(G)\to \X(\Gamma, G)/Out_T(G).$$
Consequently, $\varphi$ becomes a map
\begin{equation}\label{varphi}
\varphi:\X(\Gamma, G)/Out_T(G) \to \Ch(\Gamma,G).
\end{equation}

By \cite[Corollary 2 of Theorem 1]{Vi}, the map $\X(\Gamma, G) \to \Ch(\Gamma,G)$ is finite. Consequently, $\varphi$ is finite too.

Denote the set of $G$-representations of $\Gamma$ having Zariski dense image by $\hm^{zd}(\Gamma,G),$ and let $\X^{zd}(\Gamma, G)$ be the corresponding subspace in $\X(\Gamma, G)$. Let $\Ch^{zd}(\Gamma, G)\subset \Ch(\Gamma, G)$ be the subset of characters of Zariski dense representations of $\Gamma$, and let $\varphi^{zd}$ be the restriction of \eqref{varphi} to $\X^{zd}(\Gamma,G)/Out_T(G)\subset \X(\Gamma,G)/Out_T(G).$

The example following \cite[Proposition 8.2]{AB} shows that if $G$ is reductive, then $\hm^{zd}(\Gamma,G)$ may not be Zariski open. However, the following lemma is true.

\begin{lemma}\label{dense-rem} Let $G$ be connected and reductive.  Consider the subsets $\mathcal{S}$: $\hm^{zd}(\Gamma,G)\subset \hm(\Gamma,G),$ $\Ch^{zd}(\Gamma, G)\subset \Ch(\Gamma,G)$, $\X^{zd}(\Gamma, G)\subset \X(\Gamma, G)$, and $\X^{zd}(\Gamma,G)/Out_T(G)\subset \X(\Gamma,G)/Out_T(G).$
\begin{enumerate}
\item[(a)] If $G$ is semisimple, then the subsets in $\mathcal{S}$ are Zariski open. 
\item[(b)] $\hm^{zd}(\Gamma,G)\not=\emptyset$, whenever $\Gamma$ maps onto a free group of rank at least 2. 
\item[(c)] If $\hm^{zd}(\Gamma,G)\not=\emptyset$, then the subsets $\mathcal{S}$ are dense $($and hence Zariski dense$)$ whenever the corresponding superset in $\mathcal{S}$ is irreducible.\end{enumerate}
\end{lemma}

\begin{proof}
(a) When $G$ is semisimple $\hm^{zd}(\Gamma,G)$ is Zariski open by \cite[Proposition 8.2]{AB}. 
Let $\pi_G:\hm(\Gamma, G)\to \X(\Gamma, G)$ be the GIT quotient map.  Because $ \X(\Gamma, G)$ has the quotient topology from $\pi_G:\hm(\Gamma, G)\to \X(\Gamma, G)$ for the openness of $\X^{zd}(\Gamma, G)$ is enough to show that
\begin{equation}\label{e-pi}
\pi_G^{-1}(\X^{zd}(\Gamma, G))=\hm^{zd}(\Gamma, G).
\end{equation}
The inclusion $\supset$ is obvious. To show $\subset$, note that every Zariski dense $\rho$ is irreducible and, hence, a stable point of the action of $G$ by conjugation, by \cite[Corollary 31]{Si4}. By \cite{Do} (Theorem 8.1 and the Property (iv) of a good categorical quotient), $\pi_G^{-1}([\rho])$ is the (set-theoretic) $G$-orbit of $\rho$. Now the inclusion $\subset$ of (\ref{e-pi}) follows from the fact that all representations in that orbit (that is, conjugates of $\rho$) are Zariski dense.

The proof of openness of $\X^{zd}(\Gamma,G)/Out_T(G)$ is very similar to the above one, but simpler. Again, considering the quotient map 
$$\pi_{Out_T(G)}: \X(\Gamma,G)\to \X(\Gamma,G)/Out_T(G),$$ 
it is enough to show that 
$$\pi_{Out_T(G)}^{-1}(\X^{zd}(\Gamma,G)/Out_T(G))=\X^{zd}(\Gamma,G),$$
but since $Out_T(G)$ is finite the statement follows from the fact that this categorical quotient coincides with the set-theoretic quotient.

We now show $\Ch^{zd}(\Gamma, G)$ is open. Let $Y, Y^{zd}$ and $Y^{nzd}$ respectively denote $\X(\Gamma,G)/Out_T(G)$, $\X^{zd} (\Gamma,G)/Out_T(G)$ and the complement $$\X^{nzd}(\Gamma,G)/Out_T(G):=\X(\Gamma,G)/Out_T(G)-\X^{zd} (\Gamma,G)/Out_T(G).$$ Since $\varphi$ is onto,
$$\Ch(\Gamma,G)=\varphi(Y)=\varphi(Y^{zd})\cup\varphi(Y^{nzd}).$$
We have shown that $Y^{nzd}$ is closed and, since $\varphi$ is finite, $\varphi(Y^{nzd})$ is closed in $\Ch(\Gamma,G).$ Now the openness of $\Ch^{zd}(\Gamma,G)=\varphi(Y^{zd})$ follows from
\cite[Proposition 8.1]{AB} which implies that $\varphi(Y^{zd})$ and $\varphi(Y^{nzd})$ are disjoint.

(b) When $\Gamma$ maps onto a free group of rank at least 2, then by \cite[Lemma 1.8]{Gelander} $\hm^{zd}(\Gamma,G)$ is non-empty.  

(c)   Since $G$ is connected and reductive, $G\cong DG\times_F T$ where $DG=[G,G]$ is semisimple, $T$ is a central algebraic torus, and $F=T\cap DG$ is a finite central subgroup.  Since $\hm(\Gamma,G)$ is connected,
\begin{eqnarray*}
\hm(\Gamma,G)&\cong&\hm^0(\Gamma,DG \times T)/\hm(\Gamma,F)\\
             &\cong&(\hm^0(\Gamma,DG) \times \hm^0(\Gamma,T))/\hm'(\Gamma,F), 
\end{eqnarray*}
by \cite[Prop. 5(1)]{Si7}, where the superscript $0$ denotes the connected component of the trivial representation and $\hm'(\Gamma,F)$ denotes the subgroup of $\hm(\Gamma,F)$ mapping $\hm^0(\Gamma,DG \times T)$ to itself. A subgroup of $G\cong DG\times_F T$ is Zariski dense, if and only if it is of the form $D' \times_F T'$ for Zariski dense $D'\subset DG$, $T'\subset T.$ Consequently, $$\hm^{zd}(\Gamma,G)\cong(\hm^{0,zd}(\Gamma,DG) \times \hm^{0,zd}(\Gamma,T))/\hm'(\Gamma,F).$$ Since $\hm^{zd}(\Gamma,G)\ne \emptyset,$ $\hm^{0,zd}(\Gamma,DG)$ is non-empty and also Zariski open in $\hm^{0}(\Gamma,DG)$ by Part (a). The irreducibility of $\hm(\Gamma,G)$ implies the irreducibility of $\hm^{0}(\Gamma,DG)$ and, therefore, $\hm^{0,zd}(\Gamma,DG)$ is dense in $\hm^{0}(\Gamma,DG)$. Note that 
$$\hm^0(\Gamma, T)\cong \hm^0(\Gamma/[\Gamma,\Gamma],T)\cong T^r,$$ 
where $r$ is the rank of the free part of the abelianization $\Gamma/[\Gamma,\Gamma]$.
Since the set of elements of $T^r$ generating Zariski dense subgroups in $T^r$ is Zariski dense,
$\hm^{0,zd}(\Gamma,T)$ is Zariski dense in $\hm^{0}(\Gamma,T).$ Thus, $\hm^{zd}(\Gamma,G)$ is Zariski dense in $\hm(\Gamma,G).$ 

Similarly, $\X^{zd}(\Gamma,G)$ is dense in its superset, by the same argument coupled with \cite[Prop. 5(2)]{Si7}.
That implies $\X^{zd}(\Gamma,G)/Out_T(G)$ is dense as well. Finally, since $\varphi$ is onto, it maps the dense set $\X^{zd}(\Gamma,G)/Out_T(G)$ onto a dense set $\Ch^{zd}(\Gamma,G).$
\end{proof}

\begin{theorem}\label{irreps-thm}Let $G\subset \SL{n}$ be a semisimple group, and $\Gamma$ a finitely generated group.  Then: 
\begin{enumerate}
\item $Out_T(G)$ acts freely on $\mathcal X^{zd}(\Gamma,G)$.
\item $\varphi^{zd}:\mathcal X^{zd}(\Gamma,G)/Out_T(G)\to \Ch^{zd}(\Gamma,G)$ is a finite, birational bijection.
\item If $\mathcal X^{zd}(\Gamma,G)/Out_T(G)$ is normal, then $\varphi^{zd}$ is a normalization map.
\end{enumerate}
\end{theorem}

\bpr
(1) Suppose for $[\alpha]\in Out_T(G)$ that $[\alpha]\cdot [\rho]=[\rho]$ for some $[\rho]\in \X^{zd}(\Gamma, G)$. By the argument of the proof of Lemma \ref{dense-rem}(a), $\rho$ and $\alpha\rho$ being equivalent and Zariski dense, are conjugate.  In other words,
$\alpha(\rho(\gamma))=h\rho(\gamma)h^{-1}$ for some $h\in G$ and all $\gamma\in \Gamma.$ Since $\rho$ has Zariski dense image, $\alpha(g)=hgh^{-1}$ for all $g\in G$.  In particular, $\alpha\in Inn(G)$ and $[\alpha]$ is the identity in $Out_T(G)$.

(2) Since $\varphi^{zd}$ is a restriction of a finite map to a Zariski open set, it is also finite. It is injective by \cite[Proposition 8.1]{AB}, and $\varphi^{zd}$ is onto by definition. Then, by \cite[Lemma 1]{Vi}, $\varphi^{zd}$ is birational.  

Lastly, $(3)$ is an immediate consequence of $(2)$.
\epr

\begin{example}
Let $D$ be the diagonal matrices in $G=\SO{2n}$. Then $D$ forms an irreducible subgroup by \cite[Example 21]{Si4}. Now take any $A\in\mathrm{O}(2n)$ with $\det(A) = -1$.  Then the conjugation by $A$ is a non-trivial outer automorphism of
$\SO{2n}$ which preserves $D$. Hence, $Out_T(G)$ does not act freely on irreducible representations.  Thus, we cannot extend the previous theorem to the irreducible locus.
\end{example}

It is a general fact that bijective morphisms $f:X \to Y$ give equality $[X]=[Y]$ in the Grothendieck ring $K_0(Var/\C)$, for example see \cite[Page 115]{BBS}.  Consequently,
Theorem \ref{irreps-thm}(2) implies the following corollary.

\bcor
Under the assumptions of Theorem \ref{irreps-thm},
$[\X^{zd}(\Gamma, G)/Out_T(G)]=[\mathcal{C}h^{zd}(\Gamma, G)]$ in the Grothendieck ring $K_0(Var/\C).$
\ecor

\bthm \lb{cor-normal} Let $G\subset \SL{n}$ be a connected reductive group, and $\Gamma$ a finitely generated group.  Then: 
\begin{enumerate}
\item If $\X(\Gamma, G)/Out_T(G)$ is irreducible and contains a representation with Zariski dense image, then $\varphi$ is a birational morphism.
\item If $\X(\Gamma, G)/Out_T(G)$ is normal and contains a representation with Zariski dense image, then $\varphi$ is a normalization map.
\item $\varphi$ is a normalization map for $\Gamma$ a free group of rank $r\geq 2$ $($compare \cite{Vi}$)$.
\item $\varphi$ is a normalization map for a genus $g\geq 2$ surface group $\Gamma_g$ and $G=\SL{m}$ or $\GL{m}$, $m\geq 1$. 
\item $\varphi$ is a normalization map for $\Gamma$ a free abelian group of rank $r\geq 1$ and $G=\SL{m},$ or $\Sp{2m}$, $m\geq 1$.
\end{enumerate}
\ethm

\begin{proof}  
Since we are assuming the set of representations with Zariski dense image is non-empty and $\X(\Gamma, G)/Out_T(G)$ is irreducible as an algebraic set, $\X^{zd}(\Gamma, G)/Out_T(G)$ is dense in $\X(\Gamma, G)/Out_T(G)$ by Lemma \ref{dense-rem}(c). Now $(1)$ follows from \cite[Proposition 8.1]{AB} and \cite[Lemma 1]{Vi}.  

Then $(2)$ is an immediate consequence of $(1)$ and the finiteness of $\varphi.$

Since $\hm(F_r,G)$ is smooth, and $\X(F_r,G)$ and $\X(F_r, G)/Out_T(G)$ are subsequent GIT quotients, all three spaces are irreducible and normal. Since there is a $G$-representation of $F_r$ for any $r\geq 2$ with Zariski dense image by Lemma \ref{dense-rem}(b), $(3)$ follows from $(2)$.

By \cite{Simpson1, Simpson2},  $\X(\Gamma_g, G)$ is normal for surface groups $\Gamma_g$ for $g\geq 2$ when $G$ is $\SL{n}$ or $\GL{n}$. The epimorphism $\Gamma_g\to F_g$ together with Lemma \ref{dense-rem} yields a Zariski dense representation of $\Gamma_g\to G$ for every $g\geq 2$.  Now $(4)$ follows from $(2)$.

Lastly, by \cite{Si6} $\X(\mathbb{Z}^r, G)$ is irreducible and normal for $G$ equal to $\SL{m},$ or $\Sp{2m}$.  In these cases $Out_T(G)$ is trivial (see Section \ref{outt}), and by Proposition \ref{xt-prop} $\X(\mathbb{Z}^r, G)=\Ch(\mathbb{Z}^r, G)$.  Hence, $(5)$ follows.
\end{proof}

In particular, for $\Gamma$ and $G$ from $(1)$ in Corollary \ref{cor-normal}, regular functions on $\mathcal X(\Gamma,G)/ Out_T(G)$ can be expressed as ratios of polynomials in trace functions.

As observed above, the map $\varphi: \X(F_r, G)/Out_T(G) \to \Ch(F_r,G)$ is injective generically. However, it is not always injective on $\X(F_r, G)/Out_T(G)$, see examples in \cite{Vi}. 

On the other hand, Proposition \ref{xt-prop} shows $\varphi$ is an isomorphism for $\Gamma=F_r$ and the minimal dimensional algebraic representations of $G$ equal to $\SL{m},$ $\Sp{2m},$ or $\SO{2m+1}$. It is also an isomorphism for $G=\SO{2m}$ by \cite[Proposition 16]{Si7}.

\begin{question}  
Is $$\varphi: \X(F_r, G)/ Out_T(G)\to \Ch(F_r, G)$$ an isomorphism for any quasi-simple $G$ and any minimal dimensional algebraic representation of $G$?
\end{question}

The above question is open for groups $G$ other than those mentioned above. In particular, it is open for exceptional groups. (The example of $G=\SL{2}\times \SO{3}$ of \cite{Vi} shows the assumption of $G$ being quasi-simple is essential.)

Let $\X^{sm}(\Gamma, G)$ and $\Ch^{sm}(\Gamma, G)$ denote the smooth loci of $\X(\Gamma, G)$ and $\Ch(\Gamma, G)$, respectively.

\bcor 
Let $F_r$ be a free group of rank $r \geq 2$. If $G$ is a semisimple subgroup of $\SL{n}$, then $\X^{zd}(F_r, G)$ is \'etale equivalent to $\Ch^{zd}(F_r, G)$.
\ecor

\begin{proof}
By \cite{FLR}, $\X(F_r, G)^{zd}\subset \X^{sm}(F_r, G)$.  Since $Out_T(G)$ is a finite group acting freely on $\X^{zd}(F_r, G)$, we conclude that $\X^{zd}(F_r, G)\to \X^{zd}(F_r,G)/Out_T(G)$ is \'etale.  Since $\varphi^{zd}$ is a finite, bijective birational map, restricting it to a subspace of the smooth locus gives a local analytic isomorphism.  Hence $\X^{zd}(F_r, G)/Out_T(G)\to \Ch^{zd}(F_r, G)$ is also \'etale.  Since the composition of \'etale morphisms is \'etale, the result follows.
\end{proof}


We next analyze the groups $Out_T(G)$ for quasi-simple $G$ and show that they are often trivial.

\section{Trace preserving outer automorphisms}\label{outt}

$Out(G)$ is a subgroup of the symmetries of the Dynkin diagram of the Lie algebra of $G$ (see \cite[Appendix D]{FH}) and, hence, $Out(G)$ is finite for $G$ semisimple.  By definition, $Out_T(G)\subset Out(G)$ and, thus, $Out_T(G)$ is trivial whenever $Out(G)$ is trivial.

\begin{example}
$Out_T(G)$ is trivial for every matrix realization of odd orthogonal groups $\SO{2m+1}$, symplectic groups $\Sp{2m}$, and the complex groups $G_2, F_4, E_7,$ and $E_8$.
\end{example}

\begin{proof}
This follows from the fact that $B_n, C_n, G_2, F_4, E_7,$ and  $E_8$ have no symmetries of their Dynkin diagrams, see for example \cite{Samelson}.
\end{proof}

For $m\geq 3$, $Out(\SL{m})=\la\sigma\ |\ \sigma^2\ra$, where $\sigma$ is the Cartan involution $\sigma(A)=(A^{-1})^T$. Thus, the group $Out_T(\SL{m})$ is either trivial or equal to $Out(\SL{m})$ for $m>2$ depending on whether $\sigma$ is conjugation by an element of the normalizer of $\SL{m}$ in the ambient group $\SL{n}.$

For the canonical matrix realization of $\SL{m}$ it is easy to see that $\sigma$ is not trace-preserving for any $m>2.$  On the other hand, consider the adjoint representation of $\SL{m}$ which acts irreducibly on 
$$M_0:=\{M\in Mat(m,\C)\ |\ \tr( M)=0\}$$ 
by conjugation. It is a representation of $\mathrm{SL}(m,\C)$ into $\GL{m^2-1}$, which we denote by $\mathrm{Ad}$.

\begin{example}
For the standard matrix realization of $\SL{m}$, the group $Out_T(\SL{m})$ is trivial.  For $m\geq 3$, $$Out_T(\mathrm{Ad}(\SL{m}))\cong \Z/2\Z.$$
\end{example}

\bpr
Since $Out(\SL{m})$ is trivial for $m=2$ and $\tr(\sigma(A))\not=\tr(A)$ in general for $A\in \SL{m}$ for $m\geq 3,$ the first part follows.

For the second part, it is enough to prove that there is $P\in \mathrm{GL}(M_0)$ such that $\sigma$ coincides with conjugation by $P$ via $\mathrm{Ad},$ that is,
$$P \mathrm{Ad}(\sigma(A))P^{-1}=\mathrm{Ad}(A)$$ for every $A\in M_0.$
That means $$P(\sigma(A) (P^{-1}X) \sigma(A)^{-1})=A X A^{-1}$$ for every $X\in M_0.$
Note now that $P(M)=M^T$, obviously an invertible linear map, satisfies the above equation.
Thus, $Out_T(\mathrm{Ad}(\SL{m}))=\la \sigma \ |\ \sigma^2\ra,$ as required.
\epr

\begin{example}
For the standard matrix realization of $\SO{2m}$, the group $Out_T(\SO{2m})$ is isomorphic to $\Z/2\Z$.
\end{example}

\begin{proof}
Let $\sigma$ be an automorphism on $\SO{2m}$ given by conjugation by an orthogonal matrix of determinant $-1$ as in \cite{Si7}. Then $\sigma$ considered as an element of $Out(\SO{2m})$ does not depend on the choice of such matrix, \cite{Si7}.  Furthermore,
$\sigma^2=\mathrm{Id}$.  It is the only non-trivial outer automorphism of $\SO{2m}$ 
seen by considering the Dynkin diagram (for $m\not=4$).  Note that for $m=4$ the Dynkin diagram has order three symmetry (triality) and so $Out(\mathrm{Spin}(8,\C))$ is the symmetric group on 3 letters, but descending to $\SO{8}$ reduces the outer automorphisms back down to $\Z/2\Z$ again.  So assuming the defining matrix realization of $\SO{2m}$, $\sigma$ preserves the trace and, hence, $Out_T(\SO{2m})=\la \sigma\ |\ \sigma^2\ra$.
\end{proof}

We note that as shown in \cite{Samelson}, for the complex algebraic form of $E_6$, $Out(E_6)$ is generated by an order 2 element as well.

\begin{example}There exists a matrix realization of $E_6$ so that $Out_T(E_6)\cong \Z/2\Z$.
\end{example}

\begin{proof}
The minimal dimensional irreducible representation of $E_6$ gives an embedding of $E_6$ into $\GL{27}$.  Then the generator of $Out(E_6)$ is, as above, the Cartan involution $\sigma(A)=(A^{-1})^T$ \cite[page 68]{IY}.  As above, we can then consider the adjoint representation into $\GL{78}$ which will allow the Cartan involution to be conjugation by an element of the normalizer of $E_6$.
\end{proof}


\newcommand{\etalchar}[1]{$^{#1}$}
\def\cdprime{$''$} \def\cdprime{$''$} \def\cprime{$'$} \def\cprime{$'$}
  \def\cprime{$'$} \def\cprime{$'$}

\end{document}